\documentclass{IEEEtran}

\usepackage{cite}
\usepackage{amsmath,amssymb,amsfonts}

\usepackage{amsthm}

\usepackage{algorithmic}
\usepackage{graphicx}
\usepackage{textcomp}
\usepackage{xcolor}

\usepackage{dsfont}
\usepackage{enumitem}

\usepackage{hyperref}

\title{Velocity Stabilization of a Wave Equation with a Nonlinear Dynamic Boundary Condition%
\thanks{Nicolas Vanspranghe and Christophe Prieur are with Univ. Grenoble Alpes, CNRS, Grenoble INP, GIPSA-lab, 38000 Grenoble, France. Email: name.surname@gipsa-lab.fr.}
\thanks{Francesco Ferrante is with Department of Engineering, University of Perugia, 06125 Perugia, Italy. Email: francesco.ferrante@unipg.it.}
\thanks{This work has been partially supported by MIAI@Grenoble Alpes (ANR-19-P3IA-0003).}
}
\author{Nicolas Vanspranghe, Francesco Ferrante, Christophe Prieur}

\newtheorem{theo}{Theorem}

\newtheorem{corollary}{Corollary}
\newtheorem{lemma}{Lemma}

\theoremstyle{remark}
\newtheorem{rem}{Remark}

\begin{document}

\maketitle

\begin{abstract}
This paper deals with a one-dimensional wave equation with a nonlinear dynamic boundary condition and a Neumann-type boundary control acting on the other extremity.  We consider a class of nonlinear stabilizing feedbacks that only depend on the velocity at the controlled extremity. The uncontrolled boundary is subject to a nonlinear first-order term, which may represent nonlinear boundary anti-damping. Initial data is taken in the optimal energy space associated with the problem. Exponential decay of the mechanical energy is investigated in different cases. Stability and attractivity of suitable invariant sets are established.
\end{abstract}

\section{Introduction}

In this paper, we investigate the problem of stabilizing a one-dimensional wave equation supplied with a dynamic boundary condition {by  means} of boundary control at the opposite end {of} the domain. By definition, dynamic (or kinetic) boundary conditions involve second-order time derivative and  arise in physical problems where the boundary (or a part of it) carries its own momentum. In a  one-dimensional medium, such boundary condition is for instance obtained when considering a tip mass at one endpoint of an elastic rod as in \cite{andrews_second_1996} for wave propagation or \cite{conrad_stabilization_1998} for Euler-Bernoulli beam dynamics. In higher space dimension, vibrating membranes with a given mass density can be modeled as dynamic boundary conditions as well
-- see, e.g., \cite{goldstein_derivation_2006}, \cite{gal_oscillatory_2003}, \cite{fourrier_regularity_2013}, \cite{vitillaro_wave_2017}.

We now specify the control problem under study.  Let $L$ be a positive real number and $\Omega \triangleq (0, L)$; we consider the system
\begin{subequations}
\label{eq:pde-bc}
\begin{align}
\label{eq:pure-wave}
&\partial_{tt}u - \partial_{xx}u = 0 & &\mbox{on}~\Omega \times \mathbb{R}^+,\\
\label{eq:DBC}
&\partial_{tt}u(0, t) - \partial_{x}u(0, t) = F(\partial_t u(0, t)) & &\mbox{for all}~t,\\
\label{eq:NBC}
&\partial_{x} u(L, t) = - g(\partial_t u(L, t))  & &\mbox{for all}~t,
\end{align}
\end{subequations}
where $g$ and $F$ are (real) scalar functions satisfying the following properties:
\begin{itemize}
\item $g$ is continuous, nondecreasing, and $g(0) = 0$;
\item $F$ is globally Lipschitz continuous, and $F(0) = 0$.
\end{itemize}
Equation \eqref{eq:pure-wave} is the standard wave equation on a segment. The function $F$ in the dynamic boundary condition \eqref{eq:DBC}  represents a nonlinear behavior at the boundary $x = 0$. This term can be used to model a destabilizing boundary \emph{anti-damping} phenomenon. Equation \eqref{eq:NBC} defines a nonlinear dissipative Neumann velocity feedback.  Such boundary feedback is modeled as an \emph{unbounded} input with respect to the natural energy space of the problem, which is introduced later on.

When $F$ represents a linear boundary damping term and \eqref{eq:NBC} is replaced with a homogenous Dirichlet condition, the stability analysis of the associated semigroup of linear contractions has  been investigated in \cite{morgul_stabilization_1994} and \cite{guo_spectrum-determined_2000}. Aside from the nonlinear aspect of our work, the difference with these papers lies in the fact that the feedback control considered here is anti-collocated with respect to the dynamic boundary condition.

In the control literature,
{the coupled dynamics described by \eqref{eq:pure-wave}-\eqref{eq:DBC}, or variants,} have very often been considered in the context of minimizing torsional vibrations along drilling rods due to nonlinear friction at the rock-tip interface, where the so-called \emph{stick-slip} phenomenon may occur and destabilize the plant -- see also the review paper \cite{saldivar_control_2016} and \cite{adly_nonsmooth_2020}. System \eqref{eq:pde-bc} can be seen as an infinite-dimensional model of such plant: the rod is seen as a purely elastic medium whose angular deformation obeys the wave equation \eqref{eq:pure-wave},
and the drilling tip is subject to nonlinear torsional friction, represented by $F$ at the rock interface, which yields \eqref{eq:DBC}.
However, {in contrast with the present paper}, most of the work dealing with drilling dynamics considers linearized equations.
For instance, in \cite{terrand-jeanne_regulation_2020}, stabilization and regulation using a proportional integral boundary controller is investigated; the system is linear but the elasticity of the propagation medium is allowed to be nonhomogenous.  In \cite{mlayeh_backstepping_2018}, an observer-based boundary control design is proposed. In \cite{roman_backstepping_2018}, a backstepping-based method is considered.
{Other related works include  \cite{smyshlyaev_boundary_2009} and \cite{bresch-pietri_output-feedback_2014}, where linear first-order boundary anti-damping is considered. On the other hand,} 
nonlinear boundary feedback for distributed parameter systems are considered in \cite{hastir_well-posedness_2019} and \cite{ramirez_stabilization_2017} -- see also  \cite{coron_dissipative_2008} for stability analysis of general quasilinear hyperbolic systems with (static) {dissipative} boundary conditions.

{
This paper analyzes the stability of system \eqref{eq:pde-bc} when both the coupled boundary dynamics given by \eqref{eq:DBC}
and the velocity feedback\footnote{In the preliminary conference version of this work \cite{vanspranghe_control_2020}, only linear velocity feedback is investigated for the stabilization of \eqref{eq:pure-wave}-\eqref{eq:DBC}.
} defined by \eqref{eq:NBC} are nonlinear. 
The contributions of this work are twofold:
\begin{itemize}
\item 
The stability analysis in the presence of nonlinear anti-damping, which is the key novelty of the paper, is carried out under the assumption that the feedback nonlinearity $g$ satisfies a global sector-like condition;
\item In the particular case that $F$ is nonincreasing,
 we are able to describe the asymptotic behavior of solutions even when the sector condition only holds for large values (e.g., when $g$ represents a deadzone nonlinearity).
\end{itemize}
}

The rest of the paper is organized as follows. Section \ref{sec:wp-prem} introduces the functional settings associated with system \eqref{eq:pde-bc} and states the well-posedness of the closed-loop dynamics.
 Section \ref{sec:stability} contains the stability results along with their proofs. As mentioned above two different cases are considered:
{first, when $F$ is nonincreasing, which implies that the mechanical energy is also nonincreasing along the trajectories of the system; second, when $F$ represents an anti-damping term that may render the system unstable without feedback action.}
Some concluding remarks are given in Section \ref{sec:conc}.

\paragraph*{Notation} The norm of a given Banach space $E$ is denoted by $\| \cdot \|_E$
. If $E$ is also a Hilbert space, its scalar product is written $(\cdot, \cdot)_E$.
 Also, for $T > 0$, we denote by $W^{1, p}(0, T; E)$ the subspace of $L^p(0, T; E)$ composed of (classes of) $E$-valued functions $f$ such that, {for some $h$ in $L^p(0, T; E)$ and $\xi$ in $E$,
$
f(t) = \xi + \int_{0}^{t} h(s) \, \mathrm{d}s
$}
for a.e. $t$ in $(0, T)$. Such class $f$ is identified with its continuous representative and we say that $f' = h$ in the sense of $E$-valued distributions.
If $E$ is a metric space endowed with a distance $d$, the distance between a element $x$ and a subset $F$ of $E$ is defined as follows: $\operatorname{dist}(x, F) \triangleq \inf_{y \in F} d(x, y)$.
Finally, if $s$ is a real number, we denote by $s^+$ its positive part, i.e., $s^+ \triangleq \max \{ s, 0 \}$.
\section{Well-posedness and preliminaries}
\label{sec:wp-prem}

Problems with dynamic boundary conditions require an appropriate modification of the usual state spaces, since the boundary value $\partial_t u(0, t)$ is expected to be a continuous function with respect to the time variable. We introduce the pivot space
\begin{equation}
H \triangleq L^2(\Omega) \times \mathbb{R},
\end{equation}
which is endowed with the product Hilbertian structure: for all $\mathbf{u}_1 = (u_1, \theta_1)$ and $ \mathbf{u}_2 = (u_2, \theta_2)$ in $H$,
\begin{equation}
(\mathbf{u}_1, \mathbf{u}_2)_H \triangleq \int_\Omega  u_1(x) u_2(x) \, \mathrm{d}x + \theta_1 \theta_2.
\end{equation}
 Define now the following subset of $H$:
\begin{equation}
V \triangleq \left \{ (u, u(0)) : u \in H^1(\Omega) \right \} \simeq H^1(\Omega).
\end{equation}
Then $V$ is exactly the graph of the evaluation mapping $ u \in H^1(\Omega) \mapsto u(0)$, which is continuous. Hence, $V$ is a closed subspace of $H^1(\Omega) \times \mathbb{R}$ by the closed graph theorem, and $V$ is a Hilbert space if equipped with the inherited scalar product. It can also be proved that $V$ is a dense subspace of $H$. We consider initial data in the energy space
\begin{equation}
\mathcal{H} \triangleq V \times H,
\end{equation}
on which \eqref{eq:pde-bc} is recast into a first-order Cauchy problem having the form
\begin{subequations}
\label{eq:cauchy-pb}
\begin{align}
\label{eq:strong-diff}
&\dot{\mathbf{X}}(t) + \mathcal{A}_g(\mathbf{X}(t)) = \mathcal{F}(\mathbf{X}(t)), \\
\label{eq:strong-data}
&\mathbf{X}(0) = \mathbf{X}^0.
\end{align}
\end{subequations}
\begin{rem}
For the sake of clarity, elements of the product space $H$ are denoted using parentheses, whereas elements of $\mathcal{H}$ are denoted using brackets, as in $\mathbf{v} = (v, \theta) \in H$ and $\mathbf{X} = [\mathbf{u}, \mathbf{v}] \in \mathcal{H}$.
\end{rem}
In  \eqref{eq:cauchy-pb}, $\mathcal{A}_g$ is an unbounded nonlinear operator, with domain $\mathcal{D}(\mathcal{A}_g)$, defined by
\begin{subequations}
\begin{align}
&\mathcal{D}(\mathcal{A}_g) \triangleq \left \{ [\mathbf{u}, \mathbf{v}] \in W \times V :  \partial_x u(L) = - g(v(L)) \right \}, \\
&\mathcal{A}_g([\mathbf{u}, \mathbf{v}]) \triangleq  -[ \mathbf{v}, (\partial_{xx} u, \partial_x u(0)) ],
\end{align}
\end{subequations}
where $W \triangleq \{ \mathbf{u} \in V : u \in H^2(\Omega) \}$; and $\mathcal{F}$ is the nonlinear perturbation operator on $\mathcal{H}$ associated with $F$.
Now, we define a bilinear symmetric form $a$ on $V \times V$ as follows:
\begin{equation}
a(\mathbf{u}_1, \mathbf{u}_2) = \int_\Omega \partial_x u_1(x) \partial_x u_2(x) \, \mathrm{d}x.
\end{equation}
Also, define an energy functional $\mathcal{E}$ on $\mathcal{H}$ by
\begin{equation}
\label{eq:def-ener}
\mathcal{E}(\mathbf{u}, \mathbf{v}) \triangleq \frac{1}{2}\{ a(\mathbf{u}, \mathbf{u}) + \|\mathbf{v}\|^2_H \}.
\end{equation}
Formal computations, which motivate our choice of functional spaces, give the \emph{energy identity}
\begin{equation}
\label{eq:energy-identity-formal}
\begin{aligned}
 \frac{\mathrm{d}}{\mathrm{d}t} \mathcal{E}(\mathbf{u}(t), \mathbf{u}'(t))  =  F(\partial_t u(0, t)) \partial_t u(0, t) \\ - g(\partial_t u(L, t)) \partial_t u(L, t);
\end{aligned}
\end{equation}
as well as the \emph{variational identity}
\begin{equation}
\label{eq:variational-identity-formal}
\begin{aligned}
\frac{\mathrm{d}}{\mathrm{d}t}\{  (\mathbf{u}'(t), \mathbf{w})_H \} + a(\mathbf{u}(t), \mathbf{w}) =  F(\partial_t u(0, t))w(0) \\ - g(\partial_t u(L, t))w(L)
\end{aligned}
\end{equation}
holding for all test-functions $\mathbf{w}$ in $V$.
We shall use the classical nonlinear semigroup terminology (see, e.g., \cite[Chapter IV]{showalter_monotone_2013}):
\begin{itemize}
\item \emph{Strong} solutions to \eqref{eq:pde-bc} are absolutely continuous $\mathcal{H}$-valued functions verifying \eqref{eq:strong-diff} for a.e. $t \in \mathbb{R}^+$, with initial data in the domain $\mathcal{D}(\mathcal{A}_g)$ of the generator;
\item \emph{Weak} solutions are limits of strong solutions with respect to the topology of $\mathcal{C}([0, T], \mathcal{H})$ for a given $T > 0$, with initial data in the energy space $\mathcal{H}$.
\end{itemize}
The well-posedness properties of the closed-loop system are summarized in the following theorem. We refer the reader to \cite{vanspranghe_control_2020} for the proof.

\begin{theo}[Hadamard well-posedness]
\label{th:had-wp}
 Let $[\mathbf{u}^0, \mathbf{v}^0] \in \mathcal{H}$. Then, there exists a unique weak solution $\mathbf{u} \in \mathcal{C}(\mathbb{R}^+, V) \cap \mathcal{C}^1(\mathbb{R}^+, H)$ to \eqref{eq:pde-bc}. The following statements also hold for weak solutions:
\begin{enumerate}
\item (Hidden regularity.)
The traces $\partial_t u(L, \cdot)$, $\partial_x u(0, \cdot)$ and  $ \partial_x u(L, \cdot)$ are defined in $L^2_\mathrm{loc}(\mathbb{R}^+)$; in particular, $u(0, \cdot) \in H^2_\mathrm{loc}(\mathbb{R}^+)$ and for a.e. $t \geq 0$,
\begin{subequations}
\begin{align}
&\partial_{tt} u(0, t) - \partial_x u(0, t) = F(\partial_t u(0, t)), \\
&\partial_x u(L, t) = - g(\partial_t u(L, t));
\end{align}
\end{subequations}
\item (Energy identity.) Weak solutions satisfy \eqref{eq:energy-identity-formal} in the scalar distribution sense;
\item (Variational identity.) Weak solutions satisfy \eqref{eq:variational-identity-formal} for all $\mathbf{w} \in V$ in the scalar distribution sense;
\item (\emph{A priori} estimate.) For all $\tau \geq 0$, there exists $C(\tau) \geq 0$ (solution independent) such that
\begin{equation}
\label{eq:apriori}
\sup_{t \in [0, \tau]} \mathcal{E}(\mathbf{u}(t), \mathbf{u}'(t)) \leq C(\tau) \mathcal{E}(\mathbf{u}(0), \mathbf{u}'(0)).
\end{equation}
\end{enumerate}
Weak solutions define a strongly continuous semigroup $\{\mathcal{S}_t \}$ of (nonlinear) continuous operators acting on $\mathcal{H}$.
Also, if $[\mathbf{u}^0, \mathbf{v}^0]$ satisfies the compatibility condition
\begin{equation}
\left \{
\begin{aligned}
&u^0 \in H^2(\Omega), \mathbf{v}^0 \in V, \\
&\partial_x u^0(L) = - g(v^0(L)),
\end{aligned}
\right.
\end{equation}
then $\mathbf{u}$ is in fact a \emph{strong} solution and enjoys the following additional regularity:
\begin{equation}
u \in L^\infty_\mathrm{loc}(\mathbb{R}^+, H^2(\Omega)),  \quad \mathbf{u}' \in  L^\infty_\mathrm{loc}(\mathbb{R}^+, V).
\end{equation}
\end{theo}

\begin{rem}
It is in fact proved that unicity holds for the class of functions $\mathbf{u} \in \mathcal{C}(\mathbb{R}^+, V) \cap \mathcal{C}^1(\mathbb{R}^+, H)$ verifying the distributional identity \eqref{eq:variational-identity-formal}.
\end{rem}

\section{Stability analysis of the closed-loop system}
\label{sec:stability}

For any given weak solution $\mathbf{u}$ to \eqref{eq:pde-bc}, we denote by  $\mathcal{E}^\mathbf{u}$ the continuous function on $\mathbb{R}^+$ defined by
\begin{equation}
\mathcal{E}^\mathbf{u}(t) \triangleq \mathcal{E}(\mathbf{u}(t), \mathbf{u}'(t)).
\end{equation}
Rewriting \eqref{eq:def-ener} yields
\begin{equation}
\label{eq:ener-exp}
\mathcal{E}^\mathbf{u}(t) =  \frac{1}{2} \int_\Omega |\partial_x u|^2 + |\partial_t u|^2 \, \mathrm{d}x + \frac{1}{2}  |\partial_t u(0, t)|^2.
\end{equation}
For any $\tau \geq 0$, the integral form of \eqref{eq:energy-identity-formal} reads as follows:
\begin{equation}
\label{eq:energy-identity}
\begin{aligned}
\mathcal{E}^\mathbf{u}(t) \bigg |_0^\tau &= \int_0^\tau F(\partial_t u(0, t)) \partial_t u(0, t) \, \mathrm{d}t \\
&- \int_0^\tau g(\partial_t u(L, t)) \partial_t u(L, t) \, \mathrm{d}t.
\end{aligned}
\end{equation}
We consider two different situations:
\begin{enumerate}
\item The \emph{monotone} case where we assume that $F$ is nonincreasing, so that the energy $\mathcal{E}$ is non-increasing along the trajectories of the system;
\item The \emph{anti-damping} case, in which no prior knowledge on the sign of $F$ is assumed, meaning that the perturbation may provide energy to the system.
\end{enumerate}
{The analysis in the monotone case is built upon a multiplier method. This yields suitable integral estimates, from which the desired decay properties are deduced by taking  advantage of the nonincreasingness of the energy $\mathcal{E}$ and using an iterated sequence argument. On the other hand, }
the anti-damping case is dealt with using an appropriate Lyapunov functional {designed to exhibit the coupling between boundary perturbation at $x = 0$ and (nonlinear) dissipation at $x = L$}.  In addition, the anti-damping case is considered under both \emph{local} and \emph{global} growth assumption on the nonlinear term $F$.

\begin{rem}
All calculations performed below are justified without further comment using the additional regularity of strong solutions and the usual density arguments. In particular, for any $\tau \geq 0$, we can assume that $u \in H^2( \Omega \times(0, \tau))$; also, the additional boundary terms converge in $L^2(0, \tau)$.
\end{rem}

\subsection{The monotone case}

The first stability result of this paper is given next.

\begin{theo}[Stability in the monotone case]
\label{th:stability-monotone}
Suppose that $F$ is nonincreasing and there are some positive  constants $\alpha_1 \leq \alpha_2$ and a nonnegative $S$ such that $g$ satisfies
\begin{equation}
\label{eq:cond-sect}
\quad \alpha_1 |s| \leq |g(s)| \leq \alpha_2 |s| \quad \mbox{for all}~ s  ~\mbox{with}~ |s| \geq S,
\end{equation}
Then, there exists $E_S \geq 0$, $M >0$ and $\mu > 0$ such that, for any solution $\mathbf{u}$ to \eqref{eq:pde-bc}, for all $t \geq 0$,
\begin{equation}
\label{eq:stab-monotone}
\{ \mathcal{E}^\mathbf{u}(t) - {E}_S \}^+ \leq M \exp(-\mu t) \{ \mathcal{E}^\mathbf{u}(0) - {E}_S \}^+,
\end{equation}
where the superscript $+$ denotes the positive part.
Furthermore, if \eqref{eq:cond-sect} holds with $S = 0$, then \eqref{eq:stab-monotone} holds with ${E}_S = 0$, i.e.,
the energy $\mathcal{E}$ converges to zero exponentially.
\end{theo}
We now state a consequence of Theorem \ref{th:stability-monotone} in terms of attractive sets.
\begin{corollary}
\label{coro:attr-mono}
Under the hypotheses of Theorem \ref{th:stability-monotone}, the (positively invariant) set $\mathfrak{B}$ defined by
\begin{equation}
\mathfrak{B}\triangleq \left \{ [\mathbf{u}, \mathbf{v}] \in \mathcal{H} : \mathcal{E}(\mathbf{u}, \mathbf{v}) \leq E_S \right \}
\end{equation}
enjoys the following uniform attractivity property:
\begin{equation}
\label{eq:attr-mono}
\operatorname{dist}([\mathbf{u}(t), \mathbf{u}'(t)], \mathfrak{B})^2 \leq M' \exp(- \mu t) \{ \mathcal{E}^\mathbf{u}(0) - E_S \}^+,
\end{equation}
where $M'$ is a positive constant and $\mu$ is the decay rate appearing in Theorem \ref{th:stability-monotone}.
\end{corollary}

The rest of the subsection is devoted to the proofs of Theorem \ref{th:stability-monotone} and Corollary \ref{coro:attr-mono}.

\begin{proof}[Proof of Theorem \ref{th:stability-monotone}]
The proof is split into three steps.

\textbf{Step 1: Preliminary estimates.}
Let $\rho \in H^1(\Omega)$. Pick $\tau \geq 0$ and $\mathbf{u}$ a weak solution to \eqref{eq:pde-bc}. We multiply the wave equation
$
\partial_{tt} u - \partial_{xx} u = 0$ on $Q\tau \triangleq \Omega \times (0, T)$
by $2\rho(x) \partial_x u$, and then we integrate over $Q_\tau$:
\begin{equation}
\label{eq:wave-multiplier}
\iint_{Q_\tau} 2\rho(x) \partial_x u  \partial_{tt}u \, \mathrm{d}x \, \mathrm{d}t - \iint_{Q_\tau} 2 \rho(x) \partial_x u \partial_{xx} u  \, \mathrm{d}x \, \mathrm{d}t = 0
\end{equation}
Integrating by parts with respect to $t$  in the first term and $x$ in the second term of \eqref{eq:wave-multiplier} yields
\begin{equation}
\begin{aligned}
 \int_\Omega 2\rho(x) \partial_t u \partial_x u \, \mathrm{d}x \bigg |_0^\tau  - \iint_{Q_\tau} 2\rho(x) \partial_t u  \partial_{tx}u
\\
+ \iint_{Q_\tau} 2\left \{ \rho'(x) \partial_x u + \rho(x) \partial_{xx} u\right \} \partial_x u \, \mathrm{d}x \, \mathrm{d}t
\\ - 2 \int_0^\tau \left [ \rho(x) |\partial_x u|^2 \right ]_0^L = 0
\end{aligned}
\end{equation}
After another integration by parts with respect to the $x$,  we obtain the following standard multiplier identity:
\begin{equation}
\label{eq:mult-id}
\begin{aligned}
\int_\Omega 2\rho(x) \partial_t u \, \partial_x u\, \mathrm{d}x \bigg |_0^\tau + \iint_{Q_\tau} \rho'(x) \left \{ |\partial_t u|^2 +| \partial_x u|^2 \right \} \mathrm{d}x \, \mathrm{d}t \\
- \int_0^\tau  \left [ \rho(x)  \left \{ |\partial_t u|^2 +| \partial_x u|^2 \right \} \right ]_0^L \mathrm{d}t   = 0.
\end{aligned}
\end{equation}
From now on, we take any $\rho$ affine, positive, and increasing.
Then, we can rewrite \eqref{eq:mult-id} as
\begin{equation}
\label{eq:mult-id-bis}
\begin{aligned}
\int_0^\tau \left [ \rho(0) |\partial_t u(0, t)|^2 +  \int_\Omega \rho' \left \{ |\partial_t u|^2 + |\partial_x u|^2 \right \} \, \mathrm{d}x \right ] \mathrm{d}t
\\ = \int_\Omega 2\rho(x) \partial_t u \, \partial_x u\, \mathrm{d}x \bigg |_0^\tau  - \rho(0) \int_0^\tau |\partial_x u(0, t)|^2 \, \mathrm{d}t
\\ + \rho(L) \int_0^\tau |\partial_t u(L, t)|^2 + |\partial_x u(L, t)|^2 \, \mathrm{d}t
\end{aligned}
\end{equation}
Looking at \eqref{eq:ener-exp},  from \eqref{eq:mult-id-bis} one has that
\begin{equation}
\label{eq:est-bef-BC}
\begin{aligned}
C_1 \int_0^\tau \mathcal{E}^\mathbf{u}(t) \, \mathrm{d}t \leq C_2 \left \{ \mathcal{E}^\mathbf{u}(0) + \mathcal{E}^\mathbf{u}(\tau) \right \} \\ + C_3 \int_0^\tau  |\partial_t u(L, t)|^2 + |\partial_x u(L, t)|^2 \, \mathrm{d}t,
\end{aligned}
\end{equation}
where the Cauchy-Schwarz inequality  is used to bound the cross term in \eqref{eq:mult-id-bis}, and $C_1$, $C_2$, and $C_3$ are positive constants given by
\begin{equation}
C_1 = \min  \{  \rho(0), \rho' \}, \quad C_2 = 2\rho(L), \quad \mbox{and}~ C_3 = \rho(L).
\end{equation}
Note that the estimate \eqref{eq:est-bef-BC} holds uniformly for all solutions $\mathbf{u}$ and all times $\tau \geq 0$, with the constants depending only on the particular choice of $\rho$.
Starting from \eqref{eq:est-bef-BC}, we need to obtain an estimate where the only energy term is $\mathcal{E}^\mathbf{u}(\tau)$. Since $\mathcal{E}^\mathbf{u}$ is nonincreasing,
\begin{equation}
\label{eq:ineq-dec}
\int_0^\tau \mathcal{E}^\mathbf{u}(t) \, \mathrm{d}t \geq \tau \mathcal{E}^\mathbf{u}(\tau).
\end{equation}
From the energy identity \eqref{eq:energy-identity}, there also comes
\begin{equation}
\label{eq:ineq-var}
\mathcal{E}^\mathbf{u}(0) \leq \mathcal{E}^\mathbf{u}(\tau) + \int_0^\tau g(\partial_t u(L, t)) \partial_t u(L, t) \, \mathrm{d}t.
\end{equation}
Plugging \eqref{eq:ineq-dec} and \eqref{eq:ineq-var} into \eqref{eq:est-bef-BC} yields
\begin{equation}
\label{eq:ineq-e-tau}
\begin{aligned}
\tau C_1 \mathcal{E}^\mathbf{u}(\tau) \leq 2 C_2 \mathcal{E}^\mathbf{u}(\tau) + C_2 \int_0^\tau g(\partial_t u(L, t)) \partial_t u(L, t) \, \mathrm{d}t
\\ + C_3 \int_0^\tau  |\partial_t u(L, t)|^2 + |\partial_x u(L, t)|^2 \, \mathrm{d}t.
\end{aligned}
\end{equation}
Choosing any $\tau$ such that $\tau C_1 \geq 2 C_2 + 1$ and plugging the boundary condition into \eqref{eq:ineq-e-tau}, we finally have
\begin{equation}
\begin{aligned}
\label{eq:ineq-e-tau-bis}
\mathcal{E}^\mathbf{u}(\tau) \leq C_2 \int_0^\tau  g(\partial_t u(L, t)) \partial_t u(L, t) \, \mathrm{d}t
\\ + C_3 \int_0^\tau |\partial_t u(L, t)|^2 +  |g(\partial_t u(L,t))|^2 \, \mathrm{d}t.
\end{aligned}
\end{equation}

\textbf{Step 2: Using the boundary dissipation.}
First, consider a particular solution $\mathbf{u}$.
Similarly as in \cite{chueshov_attractor_2002}, pick a set $I_\tau$ such that
\begin{equation}
\left \{
\begin{aligned}
&|\partial_t u(L, t)| \leq S & & \mbox{for a.e.}~ t \in I_\tau, \\
&|\partial_t u(L, t)| > S  & &\mbox{for a.e.}~ t \in (0, \tau) \setminus I_\tau.
\end{aligned}
\right.
\end{equation}
This set depends on the specific solution $\mathbf{u}$. We can write
\begin{equation}
\begin{aligned}
\int_0^\tau |\partial_t u(L, t)|^2 \, \mathrm{d}t &= \int_{I_\tau}|\partial_t u(L, t)|^2 \\ &+ \int_{(0, \tau) \setminus I_\tau} [\partial_t u(L, t)|^2 \, \mathrm{d}t
\\ &\leq \frac{1}{\alpha_1} \int_0^\tau g(\partial_t u(L, t)) \partial_t u(L, t) \, \mathrm{d}t + \tau S^2.
\end{aligned}
\end{equation}
Similarly, we have
\begin{equation}
\begin{aligned}
\int_0^\tau |g(\partial_t u(L, t)|^2 \, \mathrm{d}t \leq \alpha_2  \int_0^\tau g(\partial_t u(L, t)) \partial_t u(L, t) \, \mathrm{d}t \\ + \tau \sup_{|s| \leq S} |g(s)|^2.
\end{aligned}
\end{equation}
In the end, coming back to \eqref{eq:ineq-e-tau-bis}, we obtain
\begin{equation}
\label{eq:est-aff}
\mathcal{E}^\mathbf{u}(\tau) \leq C_1 \int_0^\tau g(\partial_t u(L, t)) \partial_t u(L, t) \, \mathrm{d}t + C_2(\tau),
\end{equation}
where $C_1 > 0$ depends only on the previous constants in \eqref{eq:ineq-e-tau-bis} and the scalar function $g$, and $C_2(\tau) \geq 0$ is given by
\begin{equation}
C_2(\tau) \triangleq \tau C_3 \max \left \{S^2, \sup_{|s| \leq S} |g(s)|^2 \right \}.
\end{equation}
Note that in  \eqref{eq:est-aff} the dependence on the particular solution $\mathbf{u}$ has been suppressed. In addition, if $S = 0$, we can take $C_2(\tau) = 0$.

\textbf{Step 3: Conclusion.}
Using the energy identity \eqref{eq:energy-identity}, we get
\begin{equation}
\mathcal{E}^\mathbf{u}(\tau) \leq \frac{1}{1 + 1/C_1}\mathcal{E}^\mathbf{u}(0) + \frac{C_2(\tau)}{1 + 1/C_1}
\end{equation}
holding for any solution $\mathbf{u}$.
We can iterate: for all $k \geq 1$,
\begin{equation}
\label{eq:ari-geo}
\mathcal{E}^\mathbf{u}((k+1)\tau) \leq  \frac{1}{1 + 1/C_1}  \mathcal{E}^\mathbf{u}(k\tau) + \frac{C_2(\tau)}{1 + 1/C_1}.
\end{equation}
Let us write
\begin{equation}
r \triangleq \frac{1}{1 + 1/C_1},\quad p \triangleq \frac{C_2(\tau)}{1 + 1/C_1},\quad  \mbox{and}~ E_S \triangleq \frac{p}{1 - r}.
\end{equation}
Define a sequence $\{\mathcal{E}^\mathbf{u}_k \}_{k \geq 0}$ as follows:
\begin{equation}
\mathcal{E}^\mathbf{u}_k \triangleq \left \{
\begin{aligned}
& \mathcal{E}^\mathbf{u}(k\tau) & &\mbox{if}~ \mathcal{E}^\mathbf{u}(k\tau) \geq E_S, \\
& E_S & &\mbox{else,}
\end{aligned}
\right.
\end{equation}
so that $ \mathcal{E}^\mathbf{u}_k - E_S \geq 0$ for all $k \geq 0$.
Then,  from \eqref{eq:ari-geo} one gets
\begin{equation}
\label{eq:est-seq}
0 \leq \mathcal{E}^\mathbf{u}_k - E_S \leq r^k (\mathcal{E}^\mathbf{u}_0 - E_S)
\end{equation}
 with $ |r| < 1 $. Writing $\alpha \triangleq \ln(1/r) > 0$, \eqref{eq:est-seq} reads
\begin{equation}
\label{eq:est-seq-bis}
\mathcal{E}_k^\mathbf{u} - E_S \leq \exp(- \alpha k ) (\mathcal{E}^\mathbf{u}_0 - E_S) \quad \mbox{for all}~ k \geq 0.
\end{equation}

Now observe that
\begin{equation}
\mathcal{E}^\mathbf{u}_k - E_S = \{ \mathcal{E}^\mathbf{u}(\tau k) - E_S\}^+  \quad \mbox{for all}~ k \geq 0.
\end{equation}
On the other hand, since $\mathcal{E}^\mathbf{u}$ is nonincreasing, for all $k \tau  \leq t \leq (k+1)\tau$,
\begin{equation}
\mathcal{E}^\mathbf{u}(\tau k) - E_S \leq \mathcal{E}^\mathbf{u}(t) - E_S \leq \mathcal{E}^\mathbf{u}((k+1)\tau) - E_S
\end{equation}
Thus, using \eqref{eq:est-seq-bis}, we have
\begin{equation}
\begin{aligned}
\{ \mathcal{E}^\mathbf{u}(t) - E_S \}^+& \leq \exp (- \alpha (k+1)) \{ \mathcal{E}^\mathbf{u}(0) - E_S \}^+ \\
& \leq \exp \left [- \frac{\alpha}{\tau} \tau (k + 1) \right ]  \{ \mathcal{E}^\mathbf{u}(0) - E_S \}^+ \\
& \leq \exp (\alpha) \exp  \left [- \frac{\alpha}{\tau} t \right ]   \{ \mathcal{E}^\mathbf{u}(0) - E_S \}^+
\end{aligned}
\end{equation}
Writing
$\mu \triangleq \alpha / \tau \quad \mbox{and}~ M \triangleq \exp(\alpha)$, we finally obtain the desired uniform estimate:
\begin{equation}
\{ \mathcal{E}^\mathbf{u}(t) - E_S \}^+  \leq M \exp( - \mu t)  \{ \mathcal{E}^\mathbf{u}(0) - E_S \}^+,
\end{equation}
holding for all $t \geq 0$ and any solution $\mathbf{u}$.
\end{proof}

\begin{proof}[Proof of Corollary \ref{coro:attr-mono}] Under the hypotheses of Theorem \ref{th:stability-monotone}, the energy functional $\mathcal{E}$ is nonincreasing along the trajectories of the system. Hence, the set $\mathfrak{B}$ must be positively invariant. Equation \eqref{eq:attr-mono} is a direct consequence of Lemma \ref{lem:dist-E} in Appendix applied to the set $\mathfrak{B} = \mathfrak{B}_{E_S}$ along with estimate \eqref{eq:stab-monotone}.
\end{proof}

\subsection{The anti-damping case}

In this subsection, we deal with the anti-damping case, which is more interesting in terms of applications. The energy functional $\mathcal{E}$ is no longer nonincreasing along the trajectories of the system, which prevents the use of some of the arguments seen previously. However, on the basis of the same calculation, we can derive a proper Lyapunov functional and obtain exponential stability, at the cost of more restrictive assumptions on the nonlinear perturbation $F$ and the feedback function $g$. 

\begin{theo}[Stability in the anti-damping case]
\label{th:anti-damping}
Let $q \in (0, 1/2)$.
Assume that the feedback function $g$ satisfies
\begin{equation}
\label{eq:sector-g-bis}
\quad \alpha_1 |s| \leq |g(s)| \leq \alpha_2 |s| \quad \mbox{for all}~ s \in \mathbb{R},
\end{equation}
for some positive $\alpha_1 \leq \alpha_2$ verifying
\begin{equation}
\label{eq:cond-g}
\frac{\alpha_1}{1 + \alpha_2^2} > q.
\end{equation}
The following stability properties hold:
\begin{enumerate}
\item (Global version.)
\label{it:glocal}
 If $F$ is globally $q$-Lipschitz, then the closed-loop system is exponentially stable with respect to the energy $\mathcal{E}$, i.e., there exist positive constants $M$ and $\mu$ (solution independent) such that
\begin{equation}
\label{eq:exp-stab-AD}
\quad \mathcal{E}^\mathbf{u}(t) \leq M \exp(-\mu t) \mathcal{E}^\mathbf{u}(0) \quad \mbox{for all}~ t \geq 0;
\end{equation}
\item (Local version.)
\label{it:local}
 If $F$ is $q$-Lipschitz in some neighborhood $\mathcal{N}$ of $0$, there exists $R > 0$ such that \eqref{eq:exp-stab-AD} holds for all solutions $\mathbf{u}$ satisfying $\mathcal{E}^\mathbf{u}(0) \leq R$.
\end{enumerate}

\end{theo}
In particular, if $g = \mathrm{id}$ (proportional controller with unitary gain),  then \eqref{eq:cond-g} holds for any Lipschitz constant $q < 1/2$. Using the terminology introduced in \cite{campbell_singular_1979}, we state the counterpart of Corollary \ref{coro:attr-mono}.

\begin{corollary}
\label{coro:attr-damp}
Under the hypotheses of Theorem \ref{th:anti-damping}, assuming that $F$ is globally $q$-Lipschitz, the set
\begin{equation}
\mathfrak{A} \triangleq \{ [\mathbf{u}, \mathbf{v}] \in \mathcal{H} : \mathcal{E}(\mathbf{u}, \mathbf{v}) = 0 \},
\end{equation}
which is exactly the set of stationnary solutions, is \emph{pointwise asymptotically stable}, i.e.,
\begin{itemize}
\item Each point $\mathbf{X}$ in $\mathfrak{A}$ is Lyapunov stable;
\item Every solution $[\mathbf{u}, \mathbf{u}']$ converges in $\mathcal{H}$ to some limit in $\mathfrak{A}$.
\end{itemize}
Furthermore, the following uniform attractivity property holds:
\begin{equation}
\label{eq:attr-A}
\operatorname{dist}([\mathbf{u}(t), \mathbf{u}'(t)], \mathfrak{A})^2 \leq M' \exp( - \mu t) \mathcal{E}^\mathbf{u}(0),
\end{equation}
when $\mu$ is the decay rate given in Theorem \ref{th:anti-damping} and $M'$ is a positive constant.
 If $F$ is $q$-Lipschitz in some neighborhood of $0$, the same conclusions remain true for all solutions $\mathbf{u}$ satisfying $\mathcal{E}^\mathbf{u}(0) \leq R$.
\end{corollary}
\begin{rem}
Proof of Theorem \ref{th:anti-damping} provides an alternative method to obtain the special case of Theorem \ref{th:stability-monotone} where the constant $S$ can be taken as $0$.
\end{rem}

\begin{proof}[Proof of Theorem \ref{th:anti-damping}]
Let $\rho \in H^1(\Omega)$. We define the functional $\Gamma_\rho$ on the phase space $\mathcal{H}$ as follows:
\begin{equation}
\Gamma_\rho(\mathbf{u}, \mathbf{v}) \triangleq \mathcal{E}(\mathbf{u}, \mathbf{v}) + \int_\Omega \rho(x) \partial_x u(x) v(x) \, \mathrm{d}x.
\end{equation}
Note that $\Gamma_\rho$ is continuous on $\mathcal{H}$. Besides, if for some $\epsilon > 0$, $|\rho(x)|\leq 1 - \epsilon$ for all a.e. $x \in \Omega$, then there exist two positive constants $M_1$ and $M_2$ such that for all $[\mathbf{u}, \mathbf{v}] \in \mathcal{H}$,
\begin{equation}
\label{eq:equi-lyap}
M_1 \mathcal{E}(\mathbf{u}, \mathbf{v}) \leq \Gamma_\rho (\mathbf{u}, \mathbf{v}) \leq M_2 \mathcal{E}(\mathbf{u}, \mathbf{v}).
\end{equation}
If $\mathbf{u}$ is a given solution to \eqref{eq:pde-bc}, then we also denote by $\Gamma_\rho^\mathbf{u}$ the (continuous) function defined on $\mathbb{R}^+$ by
\begin{equation}
\Gamma_\rho^\mathbf{u}(t) \triangleq \Gamma_\rho(\mathbf{u}(t), \mathbf{u}'(t)).
\end{equation}
Then, we have
\begin{equation}
\begin{aligned}
\Gamma_\rho^\mathbf{u}(t) = \frac{1}{2} \int_\Omega |\partial_x u| + |\partial_t u|^2 + 2\rho(x) \partial_t u \partial_x u \, \mathrm{d}x \\ +\frac{1}{2} |\partial_t u(0, t)|^2.
\end{aligned}
\end{equation}

Let us write the variation of $\Gamma_\rho$ along the trajectories. For any solution $\mathbf{u}$ to \eqref{eq:pde-bc} and $\tau \geq 0$,
\begin{equation}
\label{eq:var-gamma}
\begin{aligned}
\Gamma_\rho^\mathbf{u}(t) \bigg |_0^\tau =
 - \frac{1}{2} \iint_{Q_\tau}  \rho'(x) \left \{ |\partial_t u|^2 +| \partial_x u|^2 \right \} \mathrm{d}x \, \mathrm{d}t \\
+ \int_0^\tau \left \{ F(\partial_t u(0, t)) \partial_t u (0, t) -  g(\partial_t u(L, t)) \partial_t u(L, t) \right \} \mathrm{d}t \\
+ \frac{1}{2} \int_0^\tau  \left [ \rho(x)  \left \{ |\partial_t u|^2 +| \partial_x u|^2 \right \} \right ]_0^L \mathrm{d}t.
\end{aligned}
\end{equation}
Equation \eqref{eq:var-gamma} is directly obtained summing the energy identity \eqref{eq:energy-identity} and one half of the multiplier identity \eqref{eq:mult-id}.

Just as in the proof of Theorem \ref{th:stability-monotone}, we take an affine, positive and increasing weight $\rho$. Then, \eqref{eq:var-gamma} implies
\begin{equation}
\label{eq:ineq-var-gamma}
\begin{aligned}
\Gamma_\rho^\mathbf{u}(t) \bigg |_0^\tau \leq
 - \frac{1}{2} \iint_{Q_\tau}  \rho' \left \{ |\partial_t u|^2 +| \partial_x u|^2 \right \} \mathrm{d}x \, \mathrm{d}t \\
+\int_0^\tau \left \{ F(\partial_t u(0, t)) - \frac{\rho(0)}{2}  \partial_t u(0, t)\right \} \partial_t u(0, t) \, \mathrm{d}t \\
+\frac{\rho(L)}{2}  \int_0^\tau \left \{ |\partial_t u(L, t)|^2 + |\partial_x u (L, t)| ^2 \right \} \mathrm{d}t \\
 - \int_0^\tau  g(\partial_t u(L, t) \partial_t u(L, t) \, \mathrm{d}t.
\end{aligned}
\end{equation}

\textbf{Global case.} Let us start with the global case: since $F$ satisfies $F(0) = 0$ and $F$ is $q$-Lipschitz continuous, we have
\begin{equation}
\label{eq:F-lips}
F(s)s \leq q |s|^2 \quad \mbox{for all}~ s \in \mathbb{R}.
\end{equation}
Let $\epsilon >0$ be a sufficiently small parameter to be chosen later. Then, if $\rho(0) \geq 2q + \epsilon$, using \eqref{eq:F-lips}, we see that
\begin{equation}
\label{eq:gamma-0}
\begin{aligned}
\int_0^\tau \left \{ F(\partial_t u(0, t)) - \frac{\rho(0)}{2}  \partial_t u(0, t)\right \} \partial_t u(0, t) \, \mathrm{d}t\\  \leq - \epsilon \int_0^\tau |\partial_t u(0, t)|^2 \, \mathrm{d}t.
\end{aligned}
\end{equation}
In \eqref{eq:ineq-var-gamma}, replacing $\partial_x u(L, t)$ with the boundary condition, we obtain
\begin{equation}
\label{eq:est-x-L}
\begin{aligned}
\frac{\rho(L)}{2} \int_0^\tau  \left \{ |\partial_t u(L, t)|^2 + |\partial_x u (L, t)| ^2 \right \} \mathrm{d}t
\\ = \frac{\rho(L)}{2} \int_0^\tau \left \{ |\partial_t u(L, t)|^2 + |g(\partial_t u (L, t)| ^2 \right \} \mathrm{d}t
\\ \leq \frac{\rho(L)}{2} (1 + \alpha_2^2) \int_0^\tau |\partial_t u(L, t)|^2 \, \mathrm{d}t,
\end{aligned}
\end{equation}
where we also used \eqref{eq:sector-g-bis}. Furthermore, by nondecreasingness of $g$, we have
\begin{equation}
\label{eq:g-alpha1}
\begin{aligned}
&- \int_0^\tau g(\partial_t u(L, t)) \partial_t u(L, t) \, \mathrm{d}t \\& = - \int |g(\partial_t u(L, t)| |\partial_t u(L, t)| \, \mathrm{d}t  \leq - \alpha \int_0^\tau |\partial_t u(L, t)|^2 \, \mathrm{d}t.
\end{aligned}
\end{equation}
 Then, combining \eqref{eq:est-x-L} and \eqref{eq:g-alpha1} yields
\begin{equation}
\begin{aligned}
\frac{\rho(L)}{2} \int_0^\tau  \left \{ |\partial_t u(L, t)|^2 + |\partial_x u (L, t)| ^2 \right \} \mathrm{d}t
 \\ - \int_0^\tau g(\partial_t u(L, t)) \partial_t u(L, t) \, \mathrm{d}t \\
\leq \left \{ \frac{\rho(L)}{2} (1 + \alpha_2^2) - \alpha_1 \right \} \int_0^\tau |\partial_t u(L, t)|^2 \, \mathrm{d}t.
\end{aligned}
\end{equation}
Choosing $\epsilon$ sufficiently small so that
\begin{equation}
(q + \epsilon)(1 + \alpha_2^2) \leq \alpha_1 - \epsilon, \quad \mbox{and}~ q + \epsilon \leq \frac{1}{2}(1 - \epsilon),
\end{equation}
which is possible by \eqref{eq:cond-g} along with the condition $q < 1/2$,
we observe that, if we take
\begin{equation} \rho(L) = 2q + 2\epsilon,
\end{equation}
then, on the one hand,
\begin{equation}
\label{eq:gamma-L}
\begin{aligned}
\frac{\rho(L)}{2} \int_0^\tau  \left \{ |\partial_t u(L, t)|^2 + |\partial_x u (L, t)| ^2 \right \} \mathrm{d}t
 \\ - \int_0^\tau g(\partial_t u(L, t)) \partial_t u(L, t) \, \mathrm{d}t \\
\leq - \epsilon \int_0^\tau |\partial_t u(L, t)|^2 \, \mathrm{d}t;
\end{aligned}
\end{equation}
and, on the other hand,
\begin{equation}
\rho(L) \leq 1 - \epsilon.
\end{equation}
Now, if we choose
\begin{equation}
\rho(0) = 2q + \epsilon,
\end{equation}
then, \eqref{eq:gamma-0} holds and the affine function $\rho$ uniquely defined by the choice of $\rho(0)$ and $\rho(L)$ is indeed increasing, with $\rho' = \epsilon / L$. Besides, $0 \leq \rho(x) \leq 1-\epsilon$ for all $x \in \Omega$, which guarantees that \eqref{eq:equi-lyap} holds.

With this particular choice of weight $\rho$, combining \eqref{eq:ineq-var-gamma}, \eqref{eq:gamma-0} and \eqref{eq:gamma-L}, we obtain
\begin{equation}
\Gamma^\mathbf{u}_\rho(t) \bigg |_0^\tau \leq - \mu \int_0^\tau \Gamma_\rho^\mathbf{u}(t) \, \mathrm{d}t,
\end{equation}
holding for any solution $\mathbf{u}$ and all $\tau \geq 0$, where
\begin{equation}
\mu \triangleq \min \left \{ \epsilon, \frac{\epsilon}{2L} \right \} > 0.
\end{equation}
Since $\Gamma^\mathbf{u}_\rho$ is a continuous function, an application of Gr\"onwall's lemma yields the desired estimate \eqref{eq:exp-stab-AD}.

\textbf{Local case.} Let $\tau > 0$. Recall from \eqref{eq:apriori} the \emph{a priori} estimate
\begin{equation}
\label{eq:apriori-0}
\sup_{t \in [0, {\tau}]} |\partial_t u(0, t)|^2 \leq C(\tau) \mathcal{E}^\mathbf{u}(0),
\end{equation}
where $C(\tau) > 0$ does not depend on the solution. On the other hand, we may also consider the functional $\Gamma^\mathbf{u}_\rho$, where $\rho$ is chosen as in the global case. From \eqref{eq:equi-lyap}, we have
\begin{equation}
\label{eq:a-priori-0}
\sup_{t \in [0, \tau]} |\partial_t u(0, t)|^2 \leq M_1^{-1} C(\tau) \Gamma_\rho^\mathbf{u}(0).
\end{equation}
 We infer from \eqref{eq:a-priori-0} that we can choose $R(\tau) > 0$ such that, if $\Gamma^\mathbf{u}_\rho(0) \leq R(\tau)$, then
\begin{equation}
\partial_t u(0, t) \in \mathcal{N} \quad \mbox{for all}~ t \in [0, \tau].
\end{equation}
Thus, {as in the global case}, we obtain
\begin{equation}
\Gamma_\rho^\mathbf{u}(t) \leq \Gamma^\mathbf{u}(0)\exp(-\mu t)  \quad \mbox{for all}~ t \in [0, \tau].
\end{equation}
Iterating, since for each $k \geq 1$, by \eqref{eq:apriori-0},
\begin{equation}
\begin{aligned}
\sup_{t \in [k\tau, (k+1)\tau]} |\partial_t u(0, t)|^2 &\leq M_1^{-1} C(\tau) \exp(-k\mu \tau) \Gamma_\rho^\mathbf{u}(0)
\\ &\leq R(\tau),
\end{aligned}
\end{equation}
then
\begin{equation}
\Gamma^\mathbf{u}(t) \leq \exp(-\mu t) \Gamma^\mathbf{u}(k\tau) \quad \mbox{for all}~ t \in [k\tau, (k+1)\tau].
\end{equation}
The result is now proved, with, say,
\begin{equation}
R \triangleq M_2^{-1} R(\tau).
\end{equation}
\end{proof}
{
\begin{rem}
In the proof of Theorem \ref{th:anti-damping}, deducing the local result from the global case is straightforward because of the uniform, pointwise estimate \eqref{eq:apriori-0}. This is a consequence of the velocity at $x = 0$ being a part of the state space due to the second-order boundary dynamic. However, the velocity term $\partial_t u(L, t)$ at the other endpoint can only be estimated in $L^2(0, \tau)$.

\end{rem}
}

\begin{proof}[Proof of Corollary \ref{coro:attr-damp}]
 We start with \eqref{eq:attr-A}, which is again a consequence of \eqref{eq:exp-stab-AD} and Lemma \ref{lem:dist-E} applied to the set $\mathfrak{A} = \mathfrak{B}_0$. Now, let us prove the pointwise asymptotic stability of $\mathfrak{A}$.

First, we show that each point is Lyaponov stable. Let $\mathbf{Y} = [\mathbf{w}, 0] \in \mathfrak{A}$ and $\epsilon > 0$. We must find $\delta > 0$ such that any trajectory originating from a point $\mathbf{X}^0 \in \mathcal{H}$ satisfying $\|\mathbf{X}^0 - \mathbf{Y}\|^2_\mathcal{H} \leq \delta$ must then remain in the ball $\mathcal{B}(\mathbf{Y}, \epsilon)$. Let $\mathbf{X}^0 \in \mathcal{H}$; we write $\mathbf{X}(t) = [\mathbf{u}(t), \mathbf{u}'(t)] = \mathcal{S}_t \mathbf{X}^0$. First, we have the following $H$-valued integral:
$
\mathbf{u}(t) = \mathbf{u}^0 + \int_0^t \mathbf{u}'(s) \, \mathrm{d}s$ for all $ t\geq 0
$.
Thus,
\begin{equation}
\label{eq:est-int}
\| \mathbf{u}(t) - \mathbf{u}^0\|_H \leq \int_0^\infty \| \mathbf{u}'(s) \|_H \, \mathrm{d}s \quad \mbox{for all}~ t \geq 0.
\end{equation}
The right-hand side of \eqref{eq:est-int} is finite because of the estimate
\begin{equation}
\label{eq:est-derivative}
\|\mathbf{u}'(t)\|_H \leq (2 M)^{1/2}\exp\left ( - \frac{\mu t}{2}\right ) \mathcal{E}(\mathbf{u}^0, \mathbf{v}^0)^{1/2},
\end{equation}
which comes from \eqref{eq:exp-stab-AD}. In fact, combining \eqref{eq:est-int} and \eqref{eq:est-derivative}, we obtain
\begin{equation}
\label{eq:dist-u-u0-H}
\begin{aligned}
\| \mathbf{u}(t) - \mathbf{u}^0 \|_H &\leq C \mathcal{E}(\mathbf{u}^0, \mathbf{v}^0)^{1/2} \int_0^{+\infty} \exp \left ( - \frac{\mu}{2} s \right) \, \mathrm{d}s
\\ &\leq C' \mathcal{E}(\mathbf{u}^0, \mathbf{v}^0).
\end{aligned}
\end{equation}
On the other hand, we can write
\begin{equation}
\begin{aligned}
\| \mathbf{X}(t) - \mathbf{Y}\|^2_\mathcal{H} &= \| \mathbf{u}'(t) \|^2_H + a( \mathbf{u}(t) - \mathbf{w}, \mathbf{u}(t) - \mathbf{w})
\\ &+ \| \mathbf{u}(t) - \mathbf{w} \|^2_H
\end{aligned}
\end{equation}
Since $\mathbf{Y} \in \mathfrak{A}$, $\partial_x w = 0$, and we have
\begin{equation}
\label{eq:lyap-stab-X}
\begin{aligned}
\| \mathbf{X}(t) - \mathbf{Y}\|^2_\mathcal{H} &= \| \mathbf{u}'(t) \|^2_H + a( \mathbf{u}(t), \mathbf{u}(t))
+ \| \mathbf{u}(t) - \mathbf{w} \|^2_H \\
&\leq 2 \mathcal{E}^\mathbf{u}(t) + 2 \|\mathbf{w} - \mathbf{u}^0\|^2_H + 2\|\mathbf{u}^0 - \mathbf{u}(t)\|^2_H.
\end{aligned}
\end{equation}
Plugging \eqref{eq:exp-stab-AD} and \eqref{eq:dist-u-u0-H} into \eqref{eq:lyap-stab-X}, we obtain
\begin{equation}
\| \mathbf{X}(t) - \mathbf{Y}\|^2_\mathcal{H} \leq 2\{M + (C')^2\} \mathcal{E}(\mathbf{X}^0) + 2\|\mathbf{u}^0 - \mathbf{w}\|^2_H
\end{equation}
Since the energy functional $\mathcal{E}$ is continuous on $\mathcal{H}$ and $\mathcal{E}(\mathbf{Y}) = 0$, there exists $\eta> 0$ depending only on $\mathbf{Y}$ such that if $\| \mathbf{X}^0 - \mathbf{Y}\|^2_\mathcal{H} \leq \eta$, then $\mathcal{E}(\mathbf{X}^0) \leq \epsilon$. Choosing $\delta > 0$ such that
$
\epsilon \leq \min (\eta, \epsilon)
$,
we see that $\|\mathbf{Y} - \mathbf{X}^0\|^2_\mathcal{H} \leq \delta$ implies
\begin{equation}
\|\mathbf{X}(t) - \mathbf{Y}\|^2_\mathcal{H} \leq 2\{M +(C')^2 + 1\} \epsilon \quad \mbox{for all}~ t \geq 0,
\end{equation}
which is the desired result.

Let us now prove that each trajectory $[\mathbf{u}, \mathbf{u}']$ converges (with respect to the strong topology of $\mathcal{H}$) to some limit in $\mathfrak{A}$. We already know that $\mathbf{u}'(t) \to 0$ in $H$ and $a(\mathbf{u}(t), \mathbf{u}(t)) \to 0$ as $t \to + \infty$. Pick an increasing sequence of nonnegative real numbers $t_n$ such that $t_n \to + \infty$. Then, the following estimate, which is obtained just as \eqref{eq:est-int}, holds:
\begin{equation}
\label{eq:var-seq}
\begin{aligned}
\|\mathbf{u}(t_m) - \mathbf{u}(t_n) \|_H &\leq \int_{t_n}^{+\infty} \|\mathbf{u}'(s)\|_H \, \mathrm{d}s \quad \mbox{for all}~ m \geq n.
\end{aligned}
\end{equation}
We infer from \eqref{eq:var-seq} that $\{ \mathbf{u}(t_n) \}$ is a Cauchy sequence in $H$ and thus converges to some $\mathbf{u}_\infty$. A similar argument allows us to see that $\mathbf{u}_\infty$ does not depend on the choice of the sequence $\{ t_n \}$. Then, by unicity of the limit, $\mathbf{u}_\infty$ belongs in fact to $V$ and satisfies $a(\mathbf{u}_\infty, \mathbf{u}_\infty) = 0$. Thus, $[\mathbf{u}(t), \mathbf{u}'(t)]$ converges in $\mathcal{H}$ to $[\mathbf{u}_\infty, 0] \in \mathfrak{A}$ as $t$ goes to $+ \infty$.
\end{proof}

\section{Conclusion}
\label{sec:conc}
In this paper, asymptotic stability of a one-dimensional wave equation with nonlinear boundary conditions at both endpoints of the domain has been studied. The considerered boundary conditions consist of a coupled ordinary differential equation with a nonlinear first-order term at one endpoint, and a Neumann-type nonlinear boundary dissipation at the other endpoint. When the nonlinear term $F$ does not induce an increase  of the mechanical energy and the nonlinearity $g$ can be asymptotically  lower-bounded and upper-bounded by some linear functions, it has been proved via a multiplier analysis that solutions exponentially converge to a sublevel set of the energy functional. When no sign condition is prescribed on the nonlinear function $F$, a Lyapunov-based analysis has shown that solutions converge exponentially to the set of stationary points, provided that its Lipschitz constant is small and the nonlinearity $g$ satisfies a more restrictive sector-like condition.

An interesting future outlook concerns the sharpness of the limit Lipschitz parameter $q$ obtained in Theorem \ref{th:anti-damping}. Also, we note that the class of admissible functions $g$ for the well-posedness theorem includes saturation maps; performing a stability analysis in the case of saturating feedback is a natural research line as well.

\bibliographystyle{plain}
\bibliography{wave-equation-final.bib}

\begin{thebibliography}{10}

\bibitem{adly_nonsmooth_2020}
Samir Adly and Daniel Goeleven.
\newblock A nonsmooth approach for the modelling of a mechanical rotary
  drilling system with friction.
\newblock {\em Evolution Equations \& Control Theory}, 9(4):915--934, 2020.

\bibitem{andrews_second_1996}
Kevin~T. Andrews, K.~L. Kuttler, and M.~Shillor.
\newblock Second order evolution equations with dynamic boundary conditions.
\newblock {\em Journal of mathematical analysis and applications},
  197(3):781--795, 1996.

\bibitem{bresch-pietri_output-feedback_2014}
Delphine Bresch-Pietri and Miroslav Krstic.
\newblock Output-feedback adaptive control of a wave {PDE} with boundary
  anti-damping.
\newblock {\em Automatica}, 50(5):1407--1415, 2014.

\bibitem{campbell_singular_1979}
Stephen~L. Campbell and Nicholas~J. Rose.
\newblock Singular perturbation of autonomous linear systems.
\newblock {\em SIAM Journal on Mathematical Analysis}, 10(3):542--551, 1979.

\bibitem{chueshov_attractor_2002}
Igor Chueshov, Matthias Eller, and Irena Lasiecka.
\newblock On the attractor for a semilinear wave equation with critical
  exponent and nonlinear boundary dissipation.
\newblock {\em Communications in Partial Differential Equations},
  27(9-10):1901--1951, 2002.

\bibitem{conrad_stabilization_1998}
Francis Conrad and Omer Morgul.
\newblock On the stabilization of a flexible beam with a tip mass.
\newblock {\em SIAM Journal on Control and Optimization}, 36(6):1962--1986,
  1998.

\bibitem{coron_dissipative_2008}
Jean-Michel Coron, Georges Bastin, and Brigitte d'Andréa Novel.
\newblock Dissipative boundary conditions for one-dimensional nonlinear
  hyperbolic systems.
\newblock {\em SIAM Journal on Control and Optimization}, 47(3):1460--1498,
  2008.

\bibitem{fourrier_regularity_2013}
Nicolas Fourrier and Irena Lasiecka.
\newblock Regularity and stability of a wave equation with a strong damping and
  dynamic boundary conditions.
\newblock {\em Evolution Equations \& Control Theory}, 2(4):631, 2013.

\bibitem{gal_oscillatory_2003}
Ciprian~G. Gal, Gisele~Ruiz Goldstein, and Jerome~A. Goldstein.
\newblock Oscillatory boundary conditions for acoustic wave equations.
\newblock {\em Journal of Evolution Equations}, 3(4):623--635, 2003.

\bibitem{goldstein_derivation_2006}
Gisele~Ruiz Goldstein.
\newblock Derivation and physical interpretation of general boundary
  conditions.
\newblock {\em Advances in Differential Equations}, 11(4):457--480, 2006.

\bibitem{guo_spectrum-determined_2000}
Baozhu Guo and Cheng-Zhong Xu.
\newblock On the spectrum-determined growth condition of a vibration cable with
  a tip mass.
\newblock {\em IEEE Transactions on Automatic Control}, 45(1):89--93, 2000.

\bibitem{hastir_well-posedness_2019}
Anthony Hastir, Federico Califano, and Hans Zwart.
\newblock Well-posedness of infinite-dimensional linear systems with nonlinear
  feedback.
\newblock {\em Systems \& Control Letters}, 128:19--25, 2019.

\bibitem{mlayeh_backstepping_2018}
Rhouma Mlayeh, Samir Toumi, and Lotfi Beji.
\newblock Backstepping boundary observer based-control for hyperbolic {PDE} in
  rotary drilling system.
\newblock {\em Applied Mathematics and Computation}, 322:66--78, 2018.

\bibitem{morgul_stabilization_1994}
Omer Morgul, Bo~Peng Rao, and Francis Conrad.
\newblock On the stabilization of a cable with a tip mass.
\newblock {\em IEEE Transactions on automatic control}, 39(10):2140--2145,
  1994.

\bibitem{ramirez_stabilization_2017}
Hector Ramirez, Hans Zwart, and Yann Le~Gorrec.
\newblock Stabilization of infinite dimensional port-{Hamiltonian} systems by
  nonlinear dynamic boundary control.
\newblock {\em Automatica}, 85:61--69, 2017.

\bibitem{roman_backstepping_2018}
Christophe Roman, Delphine Bresch-Pietri, Eduardo Cerpa, Christophe Prieur, and
  Olivier Sename.
\newblock Backstepping control of a wave {PDE} with unstable source terms and
  dynamic boundary.
\newblock {\em IEEE Control Systems Letters}, 2(3):459--464, 2018.

\bibitem{saldivar_control_2016}
Belem Saldivar, Sabine Mondié, S.-I. Niculescu, Hugues Mounier, and Islam
  Boussaada.
\newblock A control oriented guided tour in oilwell drilling vibration
  modeling.
\newblock {\em Annual reviews in Control}, 42:100--113, 2016.

\bibitem{showalter_monotone_2013}
Ralph~Edwin Showalter.
\newblock {\em Monotone operators in {Banach} space and nonlinear partial
  differential equations}, volume~49.
\newblock American Mathematical Soc., 2013.

\bibitem{smyshlyaev_boundary_2009}
Andrey Smyshlyaev and Miroslav Krstic.
\newblock Boundary control of an anti-stable wave equation with anti-damping on
  the uncontrolled boundary.
\newblock {\em Systems \& Control Letters}, 58(8):617--623, 2009.

\bibitem{terrand-jeanne_regulation_2020}
Alexandre Terrand-Jeanne, Vincent Andrieu, Melaz Tayakout-Fayolle, and Valerie
  Dos Santos~Martins.
\newblock Regulation of {Inhomogeneous} {Drilling} {Model} {With} a {P}-{I}
  {Controller}.
\newblock {\em IEEE Transactions on Automatic Control}, 65(1):58--71, 2020.

\bibitem{vanspranghe_control_2020}
Nicolas Vanspranghe, Francesco Ferrante, and Christophe Prieur.
\newblock Control of a {Wave} {Equation} with a {Dynamic} {Boundary}
  {Condition}.
\newblock In {\em 59th {IEEE} {Conference} on {Decision} and {Control} ({CDC}
  2020)}, Jeju Island, South Korea, 2020.

\bibitem{vitillaro_wave_2017}
Enzo Vitillaro.
\newblock On the wave equation with hyperbolic dynamical boundary conditions,
  interior and boundary damping and source.
\newblock {\em Archive for Rational Mechanics and Analysis}, 223(3):1183--1237,
  2017.

\end{thebibliography}

\appendix
\label{ap:prf}

The following lemma allows us to estimate the distance to sublevel
 sets of the energy functional $\mathcal{E}$. It is based on the Poincar\'e-Wirtinger inequality.
\begin{lemma}
\label{lem:dist-E}
Let $E \geq 0$ and
\begin{equation}
\mathfrak{B}_E \triangleq \left \{ [\mathbf{u}, \mathbf{v}] \in \mathcal{H}: \mathcal{E}(\mathbf{u}, \mathbf{v}) \leq E \right \}.
\end{equation}
Then, there exists $K > 0$ such that, for all $[\mathbf{u}, \mathbf{v}] \in \mathcal{H}$,
\begin{equation}
\operatorname{dist} ([\mathbf{u}, \mathbf{v}], \mathfrak{B}_E)^2 \leq K \{ \mathcal{E}(\mathbf{u}, \mathbf{v}) - E \}^+.
\end{equation}
\end{lemma}

\begin{proof}
Consider the continuous linear form $\Phi$ defined on $\mathcal{H}$ by
\begin{equation}
\Phi [\mathbf{u}, \mathbf{v}] \triangleq \frac{1}{L} \int_\Omega u(x) \, \mathrm{d}x,
\end{equation}
i.e., $\Phi$ gives the mean value of the position component; and recall the Poincar\'e-Wirtinger inequality:
\begin{equation}
 \|u - u_\Omega \|_{L^2(\Omega)}^2 \leq C_1 \| \partial_x u \|^2_{L^2(\Omega)} \quad \mbox{for all}~ u \in H^1(\Omega),
\end{equation}
where $u_\Omega$ denotes the mean value of $u$. Suppose that $ \mathbf{X} = [\mathbf{u}, \mathbf{v}] \in \ker \Phi$, i.e., $u_\Omega = 0$; then,
\begin{equation}
\label{eq:calc-PW}
\begin{aligned}
\|\mathbf u \|^2_V
&= \| u\|^2_{L^2(\Omega)} + |u(0)|^2 + \| \partial_x u \|^2_{L^2(\Omega)} \\
& \leq (1 + C_2) \|u \|^2_{L^2(\Omega)} + (1 + C_2) \|\partial_x u \|^2_{L^2(\Omega)} \\
& \leq (1+ C_1 + C_2 + C_1C_2) \| \partial_x u \|^2_{L^2(\Omega)},
\end{aligned}
\end{equation}
where the positive constant $C_2$ comes from the Lipschitz continuity of the evaluation mapping $u \mapsto u(0)$ with respect to the norm of $H^1(\Omega)$.
Since $\|\mathbf{X}\|^2_\mathcal{H} = \|\mathbf{u}\|^2_V + \|\mathbf{v}\|^2_H$ and
$
\mathcal{E}(\mathbf{u}, \mathbf{v}) = \frac{1}{2} \{ \|\mathbf{v}\|^2_H + a(\mathbf{v}, \mathbf{v})\}
$;
we infer from \eqref{eq:calc-PW} that there exist two positive constants $M_1$ and $M_2$ such that
\begin{equation}
\label{eq:eq-norm-ker}
M_1 \| \mathbf{X} \|^2_\mathcal{H} \leq \mathcal{E}(\mathbf{X}) \leq M_2 \| \mathbf{X} \|^2_\mathcal{H} \quad \mbox{for all}~ \mathbf{X} \in \ker \Phi.
\end{equation}
 Also, since,
$
(\ker \Phi)^\perp = \left \{ [\mathbf{u}, \mathbf{v}] \in \mathcal{H} : \mathbf{v} = 0, \partial_x u = 0  \right \}
$,
we remark that
\begin{equation}
\label{eq:invar-E}
\mathcal{E}(\mathbf{X} + \mathbf{Y}) = \mathcal{E}(\mathbf{X}) \quad \mbox{for all}~ \mathbf{X} \in \mathcal{H} ~\mbox{and}~ \mathbf{Y} \in (\ker \Phi)^\perp.
\end{equation}
Since $\ker \Phi$ is a closed subspace of $\mathcal{H}$, we can consider the orthogonal projection $\Pi$ onto $\ker \Phi$. By definition of $\mathcal{B}_E$, using \eqref{eq:invar-E}, we must have
$
\mathfrak{B}_E = \mathfrak{B}_E + (\ker \Phi)^\perp
$,
and as a consequence,
\begin{equation}
\label{eq:inv-dist}
\operatorname{dist}(\mathbf{X}, \mathfrak{B}_E) = \operatorname{dist}(\Pi \mathbf{X}, \mathfrak{B}_E) \quad \mbox{for all}~ \mathbf{X} \in \mathcal{H}.
\end{equation}
Let $\mathbf{X} = [\mathbf{u}, \mathbf{v}] \in \mathcal{H}$. We have the decomposition
$
\mathbf{X} = \Pi \mathbf{X} + [ \mathrm{id} - \Pi] \mathbf{X}
$.
First,
\begin{equation}
\begin{aligned}
\operatorname{dist}(\Pi \mathbf{X}, \mathfrak{B}_E)^2 &= \inf_{\mathbf{Y} \in \mathfrak{B}_E} \|\Pi \mathbf{X} - \mathbf{Y}\|^2_\mathcal{H}
\\
&\leq \inf_{\mathbf{Y} \in \ker \Phi \cap \mathfrak{B}_E} \|\Pi \mathbf{X} - \mathbf{Y}\|^2_\mathcal{H} \\
\end{aligned}
\end{equation}
Let $\mathbf{Y} \in \ker \Phi \cap \mathfrak{B}_E$. Then, using \eqref{eq:eq-norm-ker}
\begin{equation}
\label{eq:dist-Y}
\begin{aligned}
\operatorname{dist}(\Pi \mathbf{X}, \mathfrak{B}_E)^2 &\leq \| \Pi\mathbf{X} - \mathbf{Y}\|^2_\mathcal{H} \\
& \leq M_1^{-1} \mathcal{E}(\Pi \mathbf{X}- \mathbf{Y}).
\end{aligned}
\end{equation}
If $\mathbf{X}$ (or equivalently, $\Pi \mathbf{X}$) belongs to $\mathfrak{B}$, then the distance considered above is $0$, so we can assume that $\mathcal{E}(\mathbf{X}) > E$. Then,
\begin{subequations}
\begin{align}
&\mathbf{Y} \triangleq \{{E}/\mathcal{E}(\mathbf{X})\}^{1/2} \Pi \mathbf{X} \in \ker \Phi \cap\mathfrak{B}_E, \\
\label{eq:proj-Y}
&\mathcal{E}(\Pi \mathbf{X} - \mathbf{Y}) = \mathcal{E}(\mathbf{X}) - E.
\end{align}
\end{subequations}
Combining \eqref{eq:inv-dist}, \eqref{eq:dist-Y} and \eqref{eq:proj-Y}, we obtain the desired result:
\begin{equation}
\operatorname{dist}(\mathbf{X}, \mathfrak{B}_E)^2 \leq M_1^{-1} \{ \mathcal{E}(\mathbf{X}) - E\}^+.
\end{equation}
\end{proof}

\end{document}